\documentclass[draft]{amsart}
\usepackage{verbatim}

\usepackage{fullpage}

\numberwithin{equation}{section}
\theoremstyle{plain}
\newtheorem{theorem}[equation]{Theorem}

\newtheorem{lemma}[equation]{Lemma}

\newtheorem{corollary}[equation]{Corollary}
\newtheorem{proposition}[equation]{Proposition}

\theoremstyle{definition}
\newtheorem{definition}[equation]{Definition}
\newtheorem{remark}[equation]{Remark}

\usepackage{amsmath,amsfonts,amssymb,amscd,mathrsfs}
\usepackage[arrow,matrix,cmtip,curve]{xy}
\usepackage{enumerate, tikz,tikz-cd}

\DeclareMathOperator\Aut{Aut}

\DeclareMathOperator\End{End}
\DeclareMathOperator\Ext{Ext}

\DeclareMathOperator\GKdim{GKdim}

\DeclareMathOperator\gr{gr}

\DeclareMathOperator\Kdim{Kdim}

\DeclareMathOperator\Span{Span}
\DeclareMathOperator\stab{stab}

\newcommand\An{\widetilde{A_{n-1}}}
\newcommand{\inv}{^{-1}}
\newcommand{\iso}{\cong}
\newcommand{\kk}{{\Bbbk}}
\newcommand\p{\mathsf{p}}

\newcommand\cO{\mathcal O}
\newcommand\cS{\mathcal S}

\newcommand\CC{\mathbb C}
\newcommand\NN{\mathbb N}
\newcommand\OO{\mathbb O}

\newcommand\even{\mathrm{even}}
\newcommand\odd{\mathrm{odd}}

\title{Auslander's Theorem for dihedral actions on preprojective algebras of type A}

\author[Barahona Kamsvaag]{Jacob Barahona Kamsvaag}
\author[Gaddis]{Jason Gaddis}
\address{Miami University, Department of Mathematics, Oxford, Ohio 45056} 
\email{barahoja@miamioh.edu,gaddisj@miamioh.edu}

\subjclass[2000]{16W22,16S35}
\keywords{Auslander's theorem, preprojective algebras, Calabi-Yau algebras}

\begin{document}

\begin{abstract}
Given an algebra $R$ and $G$ a finite group of automorphisms of $R$, there is a natural map $\eta_{R,G}:R\#G \to \mathrm{End}_{R^G} R$, called the Auslander map. A theorem of Auslander shows that $\eta_{R,G}$ is an isomorphism when $R=\mathbb{C}[V]$ and $G$ is a finite group acting linearly and without reflections on the finite-dimensional vector space $V$. The work of Mori and Bao-He-Zhang has encouraged study of this theorem in the context of Artin-Schelter regular algebras. We initiate a study of Auslander's result in the setting of non-connected graded Calabi-Yau algebras. When $R$ is a preprojective algebra of type $A$ and $G$ is a finite subgroup of $D_n$ acting on $R$ by automorphism, our main result shows that $\eta_{R,G}$ is an isomorphism if and only if $G$ does not contain all of the reflections through a vertex.
\end{abstract}

\maketitle

\section{Introduction}

In \cite{QWZ}, Qin, Wang, and Zhang initiated a study of the McKay correspondence for non-connected $\NN$-graded algebras in (global) dimension two. An important component to this study is Auslander's Theorem. This project is an attempt to study this result in the context of preprojective algebras of type $A$.

Let $V$ be a finite-dimensional vector space and $G$ a finite group acting linearly on $R=\CC[V]$. The \emph{Auslander map} $\gamma_{R,G}: R\# G \to \End_{R^G} R$ is defined as 
\begin{align}
\label{eq.auslander}
a \# g &\mapsto \left(\begin{matrix}R & \to & R \\ b & \mapsto & ag(b)\end{matrix}\right).
\end{align}
Auslander's Theorem, then, states that $\eta_{R,G}$ is an isomorphism if and only if $G$ acts without reflections. That is, $G$ is a small group \cite{Au}.

The Auslander map may be defined for any algebra $R$, commutative or noncommutative, and any subgroup $G$ of $\Aut R$. However, in general the map need not be injective or surjective. Bao, He, and Zhang introduced the pertinency invariant as a tool to study Auslander's Theorem in the noncommutative setting \cite{BHZ2,BHZ1}. If $\delta$ is any dimension function on an algebra $R$, $G$ is a finite group acting on $R$, and $f_G=\sum_{g \in G} 1\# g \in R\#G$, the \emph{pertinency} of the $G$-action is defined as 
\[ \p(R,G)=\delta(R)-\delta(R\#G/(f_G)).\] 
Throughout this work, we take $\delta$ to be the Gelfand-Kirillov dimension. Under suitable hypotheses, the Auslander map is an isomorphism for the pair $(R,G)$ if $\p(R,G) \geq 2$. 

Kirkman, Moore, Won, and the second author proved that the Auslander map is an isomorphism for $(\CC_{-1}[x_1,\hdots,x_n],G)$, where $G$ is any subgroup of $\cS_n$ acting linearly as permutations of the generators (i.e., $\sigma(x_i)=x_{\sigma(i)}$) \cite{GKMW}. Chan, Young, and Zhang computed the explicit pertinency value for many cyclic subgroups of $\cS_n$ in search of noncommutative cyclic isolated singularities \cite{CYZ3}. Crawford proved that the Auslander map is an isomorphism for any pair $(R,G)$ where $R$ is a two-dimensional Artin--Schelter regular algebra and $G$ is a small group, in which ``smallness'' is generalized to the noncommutative setting using the homological determinant \cite{craw1}. Further study of Auslander's theorem and applications of the pertinency invariant can be found in \cite{CKWZ2,CKZ,GY,HZ,MU}.

A natural generalization of the above is then to study the Auslander map in the context of non-connected algebras. That is, to replace the condition of Artin--Schelter regularity with the twisted Calabi--Yau condition (see \cite{ginz}). By a result of Bocklandt, if $R$ is Calabi-Yau of dimension 2, then $R$ is the preprojective algebra of a non-Dynkin quiver \cite[Theorem 3.2]{Bock}. More generally, Reyes and Rogalski have classified graded twisted Calabi-Yau algebras of global dimension 2 that are generated in degree 1 \cite{RR1}.

We review the definition of preprojective algebras, as well as relevant ring-theoretic and homological definitions necessary to our study, in Section \ref{sec.prelim}. Our main focus, however, will be the preprojective algebra $\Pi_{\An}$, where $\An$ is the extended Dynkin diagram of type $A$. In Theorem \ref{thm.scalar}, we establish some cases where the Auslander map is an isomorphism for cyclic subgroups of scalar automorphisms on $\Pi_{\An}$.

In Section \ref{sec.dih}, we study dihedral group actions on $R=\Pi_{\An}$. Each graph automorphism of the underlying graph $\Gamma_n$ of $\An$ extends to a graded algebra automorphism of $R$. From this, we obtain a subgroup of $\Aut_{\gr}(R)$ which is isomorphic to the graph automorphism group of $\Gamma_n$, namely the dihedral group on $n$ vertices $D_n$. We identify this group with $D_n$ itself. Our main result classifies the subgroups $G$ of $D_n$ for which $\eta_{R,G}$ is an isomorphism.

\begin{theorem}
\label{thm.main}
Let $R=\Pi_{\An}$ and $G$ a subgroup of $D_n$. The Auslander map $\eta_{R,G}$ is an isomorphism if and only there exists a reflection $\tau \in D_n$ that fixes a vertex and $\tau \notin G$.
\end{theorem}
\begin{proof}
Sufficiency is proved in Theorem \ref{thm.proper}. Necessity follows from Theorems \ref{thm.Dn} and \ref{thm.Wn}.
\end{proof}

A natural next question would be to study actions on preprojective algebras associated to other extended Dynkin types. However, as these exhibit significantly fewer symmetries we do not consider them here.

\section{Preprojective algebras}
\label{sec.prelim}

Throughout, we assume that $\kk$ is an algebraically closed field of characteristic zero. All algebras are assumed to be $\kk$-algebras unless otherwise noted.

An algebra $R$ is called ($\NN$-)\emph{graded} if there exists a collection of $\kk$-vector subspaces $\{R_n\}_{n=0}^\infty$ of $R$ such that $R=\bigoplus_{n\in\NN} R_n$ and $R_iR_j\subseteq R_{i+j}$ for all $i,j\in\NN$. We say that $R$ is \emph{locally finite} if $\dim_\kk(R_n)<\infty$ for all $n\in \NN$. If $R_0=\kk$, we say that $R$ is \emph{connected}. A $\kk$-algebra automorphism $\phi$ of $R$ is called \emph{graded} if $\phi(R_n)=R_n$ for all $n\in\NN$. We denote the group of $\NN$-graded automorphisms of $R$ by $\Aut_{\gr} (R).$ 

Let $R$ be an algebra and let $G$ be a finite subgroup of $\Aut(R)$. Let $R\# G$ denote the set of formal sums
\[ R\# G = \left\lbrace \sum a_g \# g : a_g \in R, g \in G \right\rbrace.\]
Define a multiplication on $R\# G$ by 
\[ (r_1\# g_1)(r_2\# g_2)= r_1g_1(r_2)\#  g_1g_2, \quad r_i \# g_i \in R\# G,\] 
extended linearly. We call $R\# G$ the \emph{skew group ring} $R\# G$. Note that, under this definition, the group ring $\kk G$ is $\kk \#G$ with trivial $G$ action.

A \emph{quiver} $Q$ is a tuple $(Q_0,Q_1,s,t)$ consisting of a set of \emph{vertices} $Q_0=\{e_0,\dots,e_n\}$, a set of \emph{arrows} $Q_1$, and \emph{source} and \emph{target} functions $s,t:Q_1\rightarrow Q_0$. For any arrow $\alpha$, we call the vertex $s(\alpha)$ the \emph{source} of $\alpha$ and the vertex $t(\alpha)$ the \emph{target} of $\alpha$. We say $Q$ is \emph{finite} if $Q_0$ and $Q_1$ are finite, and \emph{schurian} if given any two vertices $i$ and $j$ there is at most one arrow with source $i$ and target $j$. The \emph{adjacency matrix} of a quiver $Q$ is the matrix $M_Q$ in which $(M_Q)_{ij}$ denotes the number of arrows $i \to j$.

A \emph{path of length $\ell$ with source $e_i$ and target $e_j$} is a word of the form $\alpha_1\alpha_2\cdots\alpha_\ell$ where $\alpha_k\in Q_1$ for each $k=1,\dots,\ell$, and $s(\alpha_1)=e_i$, $t(\alpha_\ell)=e_j$, and $t(\alpha_k)=s(\alpha_{k+1})$ for each $k=1,\dots,\ell-1$. If $p=\alpha_1\cdots\alpha_\ell$, we extend the source and target functions to the set of all paths by the rule $s(p)=s(\alpha_1)$ and $t(p)=t(\alpha_\ell)$.
We denote the set of paths of length $\ell$ by $Q_\ell$, and treat the vertices $e_i$ as \emph{trivial paths} of length 0 with $s(e_i)=t(e_i)=e_i$.

Let $Q$ be a quiver.
The \emph{path algebra} of a quiver $Q$ over the field $\kk$, denoted $\kk Q$, is the algebra with $\kk$-basis the set of paths $\bigcup_{\ell=0}^\infty Q_\ell$ and multiplication defined by concatenation. That is, given paths $p=\alpha_1\cdots\alpha_\ell$ and $q=\beta_1\cdots\beta_k$,
$pq=\alpha_1\cdots\alpha_\ell\beta_1\cdots\beta_k$ if $t(p)=s(q)$ and $pq=0$ otherwise.
For any $e\in Q_0$ and any path $p$, we also have the following:
\begin{align*}
    ep =\begin{cases}
    p &\quad s(p)=e\\
    0 &\quad s(p)\neq e
    \end{cases} \qquad
    pe =\begin{cases}
    p &\quad t(p)=e\\
    0 &\quad t(p)\neq e.
    \end{cases}
\end{align*}
Note that a path algebra $\kk Q$ on a finite quiver is graded by $\{(\kk Q)_\ell\}_{\ell=0}^\infty$ where $(\kk Q)_\ell=\Span_\kk(Q_\ell)$, and there are only finitely many paths of any given length, so $\kk Q$ is locally finite.

The \emph{double} of a quiver $Q$, denoted $\overline{Q}$, is defined by setting $\overline{Q}_0=Q_0$ and for every arrow $\alpha \in Q_1$ with $s(\alpha)=e_i$ and $t(\alpha)=e_j$, we add an arrow $\alpha^*$ with $s(\alpha^*)=e_j$ and $t(\alpha^*)=e_i$. We call $Q_1$ the set of \emph{nonstar arrows} and $Q_1^*:=\overline{Q_1}\setminus Q_1$ the set of \emph{star arrows}.

\begin{definition}
Let $Q$ be a finite quiver and set $\Omega=\sum_{\alpha\in Q_1} \alpha\alpha^*-\alpha^*\alpha \in \overline{Q}$. The \emph{preprojective algebra} associated to $Q$, denoted $\Pi_Q$, is the quotient $\kk \overline{Q}/(\Omega)$. We call $\Omega$ the \emph{preprojective relation}.
\end{definition}

The type $A$ extended Dynkin quivers $\An$, $n\geq 3$, are given by
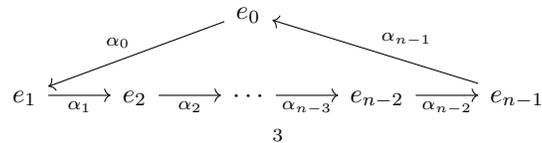
\begin{center}
\begin{tikzcd}
                           &                            & e_0 \arrow[lld, "\alpha_0"']      &                                    &                                      \\
e_1 \arrow[r, "\alpha_1"'] & e_2 \arrow[r, "\alpha_2"'] & \cdots \arrow[r, "\alpha_{n-3}"'] & e_{n-2} \arrow[r, "\alpha_{n-2}"'] & e_{n-1} \arrow[llu, "\alpha_{n-1}"']
\end{tikzcd}
\end{center}
We focus on preprojective algebras corresponding to this type.
If $Q=\An$, then $\overline{Q}$ can be characterized as follows.
The vertex set $\overline{Q_0}$ of $\overline{Q}$ is $\{e_0,\dots,e_{n-1}\}$, and there is exactly one nonstar arrow, $\alpha_i$, from $e_i$ to $e_{i+1}$ and one star arrow, $\alpha_i^*$, from $e_{i+1}$ to $e_{i}$ for each $i=0,\dots,n-1$ where the index is taken mod $n$. That is, $\overline{Q}$ is schurian. For example, the double of $\widetilde{A_2}$ is presented below:
\begin{center}
\begin{tikzcd}
                                                          & e_0 \arrow[ld, "\alpha_0"] \arrow[rd, "\alpha_2^*", shift left=2] &                                                           \\
e_1 \arrow[rr, "\alpha_1"] \arrow[ru, "\alpha_0^*", shift left=2] &                                                           & e_2 \arrow[lu, "\alpha_2"] \arrow[ll, "\alpha_1^*", shift left=2]
\end{tikzcd}
\end{center}
We note that $\widetilde{A_1}$ is also defined. However, as its double is not schurian, it does not fit into the theory we have developed.

The preprojective algebra $\Pi_{\An}$ has nice ring-theoretic properties, as discussed in the next proposition.
Before this, we review some of the definitions. 
For others, such as \emph{graded injectively smooth} and the \emph{generalized Gorenstein} condition, we refer to \cite{SZ} and \cite{SW2}, respectively.

Let $R = \bigoplus_{n \in \NN} R_n$ be a locally finite graded algebra.
The \emph{Gelfand--Kirillov (GK) dimension} of $R$ is defined as
\[ \GKdim (R) = \overline{\lim_{n\to\infty}}\log_n\left(\sum_{i\leq n}\dim_\kk( R_i) \right).\] 
Now let $\delta$ be a dimension function on $R$. For example, we may have $\delta=\GKdim$ or $\delta=\Kdim$, the Krull dimension. The ring $R$ is \emph{$\delta$-Cohen--Macaulay} ($\delta$-CM) if $\delta(R)$ is finite, and for every nontrivial finitely generated right $R$-module $M$, $j(M)+\delta(M) = \delta(R)$,
where $j(M)=\min\{i:\Ext_R^i (M,R)\neq 0\}$ denotes the \emph{grade} of $M$.

The \emph{(total) Hilbert series} of $R$ is the formal power series
\[ H_R^{tot} = \sum_{k=0}^\infty \dim_{\kk}(R_k) t^k.\]
If $R_0 = \kk^n$ with primitive idempotents $\{e_0,\hdots,e_{n-1}\}$, then the \emph{matrix-valued Hilbert series} is the matrix with entries
\[ (H_R)_{ij} = \sum_{k=0}^\infty \dim_{\kk}(e_{i+1} R_k e_{j+1} ) t^k.\]

\begin{proposition}
\label{prop.props}
Let $R=\Pi_{\An}$. Then $R$ is a locally finite graded noetherian algebra of global and GK dimension 2. Moreover, $R$ is $\GKdim$-CM and $H_R = (1-M_{\overline{Q}} t)^{-2}$.
\end{proposition}
\begin{proof}
That $R$ is locally finite graded follows because $Q$ (and hence $\overline{Q}$) is finite. By \cite[Theorem 6.5]{BGL}, $R$ is a noetherian polynomial identity (PI) ring. Further, $R$ satisfies the generalized Gorenstein condition \cite[Proposition 2.11]{SW1}, which implies that $R$ is right graded injectively smooth. Then by \cite[Theorem 1.3]{SZ}, $R$ is $\Kdim$-CM which implies that $R$ is $\GKdim$-CM by \cite[Lemma 4.3]{SZ}. The statement on the Hilbert series follows from \cite{Bock} and \cite{RR1}.
\end{proof}

The next lemma gives a canonical form for paths in the preprojective algebra $\Pi_{\An}$.

\begin{lemma}[Structure Lemma]
\label{lem.structure}
Let $Q=\An$ and let $R=\Pi_Q$. Let $p$ be a (nonconstant) monomial in $R$.
Then there exist nonstar arrows $\beta_1,\dots,\beta_\ell$ and star arrows $\gamma_1,\dots,\gamma_m$ such that $p=\beta_1\cdots\beta_\ell\gamma_1\cdots\gamma_m$.
\end{lemma}
\begin{proof}
Given any idempotent $e_i \in Q_0$, $e_i\Omega e_i \in (\Omega)$ and so the following relations hold in $\Pi_{\An}$:
\begin{align}
\label{eq.relns}
0 = e_i\Omega e_i = \alpha_i \alpha_i^* - \alpha_{i-1}^*\alpha_{i-1}
\end{align}
where the indices are taken mod $n$.

Given any star arrow $\alpha_i^*$, the only nonstar arrow $\beta$ such that $\alpha_i^*\beta\neq 0$ is $\alpha_i$.
Consequently, whenever a star arrow is followed by a nonstar arrow, we can use \eqref{eq.relns} to obtain a nonstar arrow followed by a star arrow instead. By repeated use of \eqref{eq.relns} we have
\begin{align*}
    \alpha_j^*\alpha_j\alpha_{j+1}\cdots\alpha_{j+k} &= \alpha_{j+1}\alpha_{j+1}^*\alpha_{j+1}\alpha_{j+2}\cdots\alpha_{j+k}\\
    &=\alpha_{j+1}\alpha_{j+2}\alpha_{j+2}^*\alpha_{j+2}\alpha_{j+3}\cdots\alpha_{j+k}\\
    &\ \vdots\\
    &=\alpha_{j+1}\cdots\alpha_{j+k}\alpha_{j+k}^*\alpha_{j+k}\\
    &=\alpha_{j+1}\cdots\alpha_{j+k+1}\alpha_{j+k+1}^*.
\end{align*}
By induction on the number of star arrows, it follows that we can push all star arrows to the right.
\end{proof}

The invariant theory of preprojective algebras was studied by Weispfenning, with particular interest towards a version of the Shephard--Todd--Chevalley theorem \cite{SW1,SW2}. Our interest is in a version of Auslander's Theorem for group actions on the projective algebra $R=\Pi_{\An}$. Particularly relevant to the present investigation is the following theorem due to Bao, He, and Zhang.

\begin{theorem}[{\cite[Theorem 3.5]{BHZ1}}]
\label{thm.bhz}
Let $R$ be a Noetherian locally-finite graded algebra and $G$ a finite subgroup of $\Aut_{\gr}(R)$. Assume further that $R$ is $\GKdim$-CM of global dimension 2 with $\GKdim R\geq 2$. Then $\eta_{R,G}$ is a graded algebra isomorphism if and only if $\p(R,G)\geq 2$.
\end{theorem}

Let $R'$ be the image of $R$ in the composition
\[ R \hookrightarrow R \# G \to (R\# G)/(f_G).\]
We call $R'$ the \emph{identity component} of $(R\# G)/(f_G)$, which we can associate with $R$.
By \cite[Lemma 5.2]{BHZ1}, $\GKdim(R') = \GKdim((R\# G)/(f_G))$.

\begin{theorem}
\label{thm.aus}
Let $n\geq 3$, let $R=\Pi_{\An}$, let $G$ be a finite subgroup of $\Aut_{\gr}(R)$, and let $R'$ be the identity component of $(R\# G)/(f_G)$. Then $\eta_{R,G}$ is a graded algebra isomorphism if and only if $\dim_\kk(R')<\infty$.
\end{theorem}
\begin{proof}
By Proposition \ref{prop.props}, $R$ satisfies the conditions of Theorem \ref{thm.bhz}. That is, $\eta_{R,G}$ is an isomorphism if and only if $\p(R,G)\geq 2$. Since $\GKdim(R)=2$, then $\p(R,G)\geq 2$  if and only if $\GKdim ((R\# G)/(f_G))=0$, which is equivalent to $(R\# G)/(f_G)$ being a finite dimensional $\kk$-vector space.
\end{proof}

Since $R$ is locally finite, $\kk\cup \bigcup_{\ell< m} R_\ell$ is finite dimensional for all $m\in\NN$. Consequently, if there exists $m\in \NN$ such that every path of length at least $m$ is in $(f_G)$, then $R'$ has a complete set of coset representatives in $\kk\cup\bigcup_{\ell< m} R_m$. That is, $\dim_\kk(R')<\infty$.

We conclude this section with a discussion of graded automorphisms of $\Pi_{\An}$.

Let $Q$ be a quiver. A \emph{quiver automorphism} $\sigma=(\sigma_0, \sigma_1)$ of $Q$ is a pair of bijections $\sigma_0:Q_0\to Q_0$ and $\sigma_1:Q_1\to Q_1$ such that for all $\alpha\in Q_1$, $s(\sigma_1(\alpha))=\sigma_0(s(\alpha))$ and $t(\sigma_1(\alpha))=\sigma_0(t(\alpha))$.
Every quiver automorphism extends to a graded automorphism of $\kk Q$ which, by an abuse of notation, we again denote by $\sigma$ (see \cite{KW}). 

Conversely, if $\sigma$ is a graded automorphism of $\Pi_{\An}$, then $\sigma_0$ permutes the set $Q_0$. Since $Q$ is schurian, then for any $\alpha \in Q_1$, $\sigma(\alpha)$ is necessarily a nonzero scalar multiple of the unique arrow from $\sigma_0(s(\alpha))$ to $\sigma_0(t(\alpha))$. 
First, we will consider automorphisms which fix the vertices of $\An$. In Section \ref{sec.dih}, we study automorphisms corresponding to dihedral automorphisms on $\An$. 

For the remainder of this section, let $Q=\An$ and $R=\Pi_{\An}$. Let 
\[ F=\{\sigma\in\Aut_{\gr}(R): |\sigma|<\infty \text{ and } \sigma(e_i)=e_i\text{ for all }i=0,\dots,n-1\}.\] 
Fixed subrings of $R$ under automorphisms in $F$ were studied by Weispfenning \cite{SW1,SW2}. 

Let $\sigma \in F$. By the above discussion, $\sigma(\alpha) \in \Span_\kk\{\alpha\}$ for each $\alpha \in Q_1$. Thus, there exists $\xi_i,\xi_i^* \in \kk^\times$ such that $\sigma(\alpha_i)=\xi_i \alpha_i$ and $\sigma(\alpha_i^*)=\xi_i^* \alpha_i^*$ for each $i=0,\hdots,n-1$. Since $\sigma$ is of finite order, each $\xi_i,\xi_i^*$ must be a root of unity. On the other hand, $\sigma(\Omega) \in \Span_{\kk}(\Omega)$ and so $\sigma(\Omega)=\omega \Omega$ for some $\omega \in \kk^\times$. It is not difficult to show using the preprojective relation that $\xi_i\xi_i^*=\omega$ for all $i$. The value $\omega$ in this case is the \emph{homological determinant} of the $\sigma$-action on $R$ \cite{gourmet,SW2}. We will consider cases in which $\omega=1$.

Our primary tool for studying the Auslander map for cyclic subgroups of $F$ is the following result of He and Zhang, which we have rephrased for our purpose.

\begin{lemma}[{\cite[Lemma 3.4]{HZ}}]
\label{lem.seq}
Let $\sigma\in F$, let $G=\langle\sigma\rangle$, and and let $|\sigma|=m$. Assume there are elements $a_0,\dots,a_{m-1}\in R$ such that $\sigma(a_i)=\zeta a_i$ for $i=0,\dots,m-1$, where $\zeta$ is an $m$th primitive root of unity. Then
$a_0a_1\cdots a_{m-1}\# 1 \in (f_G)$.
\end{lemma}

We now apply Lemma \ref{lem.seq} to establish an isomorphism of the Auslander map for certain scalar automorphisms.

\begin{theorem}
\label{thm.scalar}
Let $\sigma \in F$ with $m=|\sigma|$, $1 < m < \infty$, and let $G=\langle \sigma \rangle$. As above, write $\sigma(\alpha_i)=\xi_i \alpha_i$ and $\sigma(\alpha_i^*)=\xi_i^* \alpha_i^*$, $i=0,\hdots,n-1$, with $\xi_1\xi_1^*=1$.
In each of the following cases, $\eta_{R,G}$ is an isomorphism.
\begin{enumerate}
    \item \label{sc1} There is some primitive $m$th root of unity $\zeta$ such that $\xi_i=\zeta$ for $i=0,\hdots,n-1$.
    \item \label{sc2} There is some primitive $m$th root of unity $\zeta$ such that $\xi_0\xi_1\cdots\xi_{n-1}=\zeta$.
    \item \label{sc3} For all $i,j=0,\hdots,n-1$ with $i\neq j$, we have $\gcd(|\xi_i|,|\xi_j|)=1$.
\end{enumerate}
\end{theorem}

\begin{proof}
First, suppose there exists a pure nonstar path $q$ of length $\ell$ such that $q\# 1 \in (f_G)$. Then let $p$ be a path containing at least $2\ell$ nonstar arrows. Pushing all nonstar arrows to the left using the Structure Lemma (Lemma \ref{lem.structure}), it follows that $p$ contains $q$, so $p\# 1 \in (f_G)$. The same argument applies if $q,p$ are pure star paths. Hence, if $(f_G)$ contains both a pure nonstar path and a pure star path of length $\ell$, then $\dim_{\kk}(R')<\infty$ and $\eta_{R,G}$ is an isomorphism by Theorem \ref{thm.aus}. Thus, in each case we will attempt to produce such paths.

(1) Assuming $\xi_i=\zeta$ as in the hypothesis, take $a_i=\alpha_i$ and apply Lemma \ref{lem.seq}. Taking indices mod $n$, it follows that $a_0\cdots a_{m-1} \in (f_G)$ is a pure nonstar path. One similarly obtains a pure star path. Hence, $\eta_{R,G}$ is an isomorphism by the above argument.

(2) Let $p$ be any pure nonstar path of length $n$, and $q$ any pure star path of length $n$. Since $p$ contains each nonstar arrow exactly once, $\sigma(p)=\zeta p$. No power of $p$ is zero and so we apply Lemma \ref{lem.seq} with $a_i=p$ to obtain $p^m\# 1 \in (f_G)$. Similarly we obtain $q^m\#1 \in (f_G)$. Hence, $\eta_{R,G}$ is an isomorphism.

(3) The order of $\sigma$ is determined by its image on $R_1$ which in turn is determined by its image on $\alpha_0,\dots,\alpha_{n-1}$. That is, $|\sigma|=k$, where $k$ is the least positive integer such that $\sigma^k(\alpha_i)=\alpha_i$ for all $i=0,\dots,n-1$. Since the orders of the scalars $\xi_i$ are relatively prime, then we have
$|\zeta| = |\xi_0|\cdot|\xi_1|\cdots|\xi_{n-1}| = |\sigma|$.
The result now follows from \eqref{sc2}.
\end{proof}

\section{Dihedral actions on $\Pi_{\An}$}
\label{sec.dih}

In this section we establish our main theorem regarding the Auslander map for dihedral actions on $\Pi_{\An}$. 

Let $\sigma$ be a quiver automorphism of a schurian quiver $Q$. As discussed above, if $\alpha \in Q_1$, then $\sigma(\alpha)$ is a scalar multiple of the unique arrow from $\sigma_0(s(\alpha))$ to $\sigma_0(t(\alpha))$. Throughout this section, we assume that scalar is 1.

\begin{proposition}
\label{prop.lift}
Let $Q$ be a quiver such that $\overline{Q}$ is schurian, let $R=\Pi_Q$, and let $\sigma\in\Aut_{\gr}(\kk \overline{Q})$ be induced from a quiver automorphism of $\overline{Q}$ as above. If one of the following hold, then $\sigma\in\Aut_{\gr}(R)$:
\begin{enumerate}
    \item $\sigma$ is \emph{star-preserving}: $\sigma(Q_1)=Q_1$ and $\sigma(Q_1^*)=Q_1^*$;
    \item $\sigma$ is \emph{star-inverting}: $\sigma(Q_1)=Q_1^*$ and $\sigma(Q_1^*)=Q_1$.
\end{enumerate}
\end{proposition}
\begin{proof}
Given a nonstar arrow $\alpha$ with source $e_i$ and target $e_j$, $\alpha^*$ is the unique arrow with source $e_j$ and target $e_i$. In particular, this holds for $\sigma(\alpha)$, so in case 1 we must have $\sigma(\alpha^*)=\sigma(\alpha)^*$. Then $\sigma(\alpha\alpha^*)=\sigma(\alpha)\sigma(\alpha)^*$ and $\sigma(\alpha^* \alpha)=\sigma(\alpha)^*\sigma(\alpha)$,
so $\sigma$ permutes the summands of both $\sum_{\alpha \in Q_1} \alpha\alpha^*$ and $\sum_{\alpha \in Q_1} \alpha^* \alpha$. That is,
\[ \sigma\left(\Omega\right)=\sigma\left(\sum_{\alpha \in Q_1} \alpha\alpha^*\right) -\sigma\left(\sum_{\alpha \in Q_1} \alpha^* \alpha\right)= \sum \alpha\alpha^* - \sum \alpha^*\alpha = \Omega.\]
The argument is similar in case 2 except we obtain $\sigma(\Omega)=-\Omega$, so again it preserves the ideal $(\Omega)$.
\end{proof}

For the remainder of this section, let $Q=\An$ and $R=\Pi_{\An}$.  We will show that there is a group of quiver automorphisms of $\overline{Q}$ that is isomorphic to the dihedral group on $n$ vertices. We first identify two quiver automorphisms of $\overline{Q}$ which extend to automorphisms of $R$.
\begin{enumerate}
\item Define $\rho:\overline{Q}\to\overline{Q}$ by $\rho(e_i)=e_{i+1}$, where the index is taken mod $n$.
Then $\rho(\alpha_i)$ is the unique arrow with source $\rho(s(\alpha_i))=e_{i+1}$ and target $\rho(t(\alpha_i))=e_{i+2}$, which is $\alpha_{i+1}$. Consequently $\rho(\alpha_i^*)=\rho(\alpha_i)^*=\alpha_{i+1}^*$.
Thus $\rho$ is a star-preserving automorphism of $\overline{Q}$, and has order $n$.

\item Define $r:\overline{Q}\rightarrow\overline{Q}$ by $r(e_i)=e_{n-i}$. Since $r(s(\alpha_i))=e_{n-i}$ and $r(t(\alpha_i))=e_{n-i-1}$, we must have $r(\alpha_i)=\alpha_{n-i-1}^*$ and $r(\alpha_i^*)=\alpha_{n-i-1}$. Thus $r$ is a star-inverting automorphism of $\overline{Q}$ order $2$.
\end{enumerate}

By Proposition \ref{prop.lift}, $G=\langle \rho, r\rangle$ extends to a subgroup of $\Aut_{\gr}(R)$ where $R=\Pi_Q$. It is clear that $G \iso D_n$, and so we identify $D_n$ with the group $G$ acting on $R$ by graded automorphisms. 

\begin{theorem}
\label{thm.proper}
Let $G$ be a subgroup of $D_n$. If there exists a reflection $\tau \in D_n$ that fixes a vertex and $\tau \notin G$, then $\dim_\kk (R')<\infty$.
\end{theorem}
\begin{proof}
Let $\tau$ be the reflection that fixes $e_i$ and suppose $\tau \notin G$. Since $\tau$ is the only nontrivial element of $D_n$ that fixes $e_i$,
we have $e_i g(e_i)=0$ for all $g\in G\setminus\{1\}$. Consequently $e_i(f_G) e_i=e_i\# 1$. Let $p$ be a path of length at least $2n+1$, so $p$ contains at least $n+1$ nonstar arrows or at least $n+1$ star arrows. Without loss of generality, suppose $p$ has at least $n+1$ nonstar arrows. By the Structure Lemma (Lemma \ref{lem.structure}), we may push all star arrows to the right, so that 
\[ p=\alpha_j\alpha_{j+1}\cdots \alpha_{j+n-1}\alpha_{j+n}p'\]
for some path $p'$ and some $j=0,\dots,n-1$ where the indices are taken mod $n$. Then for some $0\leq k\leq n-1$, $i=j+k+1$ mod $n$, so
\[ p=(\alpha_j\cdots\alpha_{j+k}) e_i (\alpha_{j+k+1}\cdots\alpha_{j}p')  \]
Hence $p\# 1\in (f_G)$ and so  $q\# 1 \in (f_G)$ for all paths $q$ of length at least $2m+1$.
Thus, $\dim_\kk (R')<\infty.$
\end{proof}

Theorem \ref{thm.proper} shows that the Auslander map is an isomorphism for the pair $(R,G)$ so long as $G$ is missing a reflection which fixes some vertex. In case $n$ is odd, this includes all proper subgroups of $D_n$. However, in the case that $n$ is even, there is one additional subgroup, $W_n$. It remains to show that the Auslander map fails to be an isomorphism in the case of $W_n$ and the full dihedral group $D_n$.

\subsection{The $D_n$ case}

For $x \in R$, we denote by the $\cO(x)$ the orbit of $x$ under $D_n$. We begin by describing the orbits of $R$ under the $D_n$ action so as to find a $\kk$-basis of $R^{D_n}$.

Recall that for $k \geq 0$, we let $Q_k$ (resp. $Q_k^*$) denote the set of paths of length $k$ containing only nonstar (resp. star) arrows, and $Q_0=Q_0^*$ is the set of trivial paths. Further, let $Q_\ell Q_k^*$ denote the set of paths containing exactly $\ell$ nonstar arrows followed by $k$ star arrows. Then in the double quiver, we have 
\[ \overline{Q_\ell} = \bigcup_{\substack{i+j=\ell \\ i,j\geq 0}} Q_iQ_j^*.\]
Clearly, $\overline{Q_\ell}$ is a generating set for the graded piece $R_\ell$ of $R$. Finally, for $\ell \geq k \geq 0$, set $B_{\ell,k}=Q_\ell Q_k^*\cup Q_k Q_\ell^*$.

\begin{lemma}
\label{lem.Blk}
For any $p\in B_{\ell,k}$, $\displaystyle \cO(p)= B_{\ell,k}$. 
\end{lemma}

\begin{proof}
By Structure Lemma (Lemma \ref{lem.structure}), a path $p$ is uniquely determined by its source along with the number of nonstar and star arrows it contains. Consequently, each $p\in Q_\ell Q_k^*$ is uniquely determined by its source, as is each $q\in Q_k Q_\ell^*$. Thus for each $i=0,\dots,n-1$, let $p_i$ (resp. $q_i$) denote the unique path in $Q_\ell Q_k^*$ (resp. $Q_k Q_\ell^*$) with source $e_i$. Then $B_{\ell,k}=\{p_0,\dots,p_{n-1},q_0,\dots,q_{n-1}\}$.

Let $x\in B_{\ell,k}$ and $y\in\cO(x)$, so $y=g(x)$ for some $g\in D_n$. If $g$ is a rotation, then $g$ bijectively maps $Q_1$ to $Q_1$ and $Q_1^*$ to $Q_1^*$.
Consequently $g(x)$ has the same number of nonstar arrows as $x$, and the same number of star arrows as $x$. That is, if $x\in Q_\ell Q_k^*$, then $y\in Q_\ell Q_k^*$. Thus $y\in B_{\ell,k}$.
If $g$ is a reflection, then $g$ bijectively maps $Q_1$ to $Q_1^*$ and $Q_1^*$ to $Q_1$.
Hence $g(x)$ has the same number of nonstar arrows as $x$ has star arrows, and the same number of star arrows as $x$ has nonstar arrows. That is, if $x\in Q_\ell Q_k^*$, then $y\in Q_k Q_\ell^*$.
Once again $y\in B_{\ell,k}$, so $\cO(x)\subseteq B_{\ell,k}.$

We have $|\cO(x)|=|D_n|/|\stab(x)|,$ and $g\in\stab(x)$ only if $g$ fixes the source of $x$.
Hence $g$ is the identity or the unique reflection $r$ fixing $s(x)$.
Now, if $\ell\neq k$, then $Q_\ell Q_k^*\neq Q_k Q_\ell^*$ so $|B_{\ell,k}|=2n$, and $r$ inverts the number of star and nonstar arrows, so $r(x)\neq x$.
Consequently $\stab(x)=\{1\}$, so $|\cO(x)|=|G|=2n$.
Thus $\cO(x) = B_{\ell,k}$.
If on the other hand $\ell=k,$ then $B_{\ell,k}=B_{\ell,\ell}=Q_\ell Q_\ell^*$, and $r(x)=x$ so $\stab(x)=\{1,r\}$.
In this case we have $|B_{\ell,k}|=n=|D_n|/|\stab(x)|=|\cO(x)|$, so $\cO(x)=B_{\ell,k}$.
\end{proof}

Set \[\OO(\ell,k) = \sum_{p \in B_{\ell,k}} p.\] By Lemma \ref{lem.Blk}, these are exactly the orbit sums of homogeneous elements in $R$, and hence form a $\kk$-basis for $R^{D_n}$. This shows that $R^{D_n}$ has Hilbert series 
\begin{align*}
H_{R^{D_n}}(t) 
	&= 1 + t + 2t^2 + 2t^3 + 3t^4 + 3t^5 + \cdots \\
	&= (1+t)\sum_{k=0}^\infty (k+1) (t^2)^k
	= \frac{1}{(1-t)(1-t^2)}.
\end{align*}

\begin{lemma}
\label{lem.osums}
The orbit sums $\OO(\ell,k)$ satisfy the following relations:
\begin{align}
\label{eq.O1} 
\OO(1,0)\OO(\ell,k) &= 
\begin{cases}
	\OO(\ell+1,k) + \OO(\ell,k+1) & \text{ if $\ell > k$ } \\
	\OO(\ell+1,k) & \text{ if $\ell=k$}
\end{cases} \\
\label{eq.O2} \OO(1,1)^m &= \OO(m,m).
\end{align}
\end{lemma}
\begin{proof}
To prove \eqref{eq.O1}, we suppose that $\ell>k$ and then
\begin{align*}
&\OO(1,0)\OO(\ell,k) \\ 
&= \left( \sum_{i=0}^{n-1} \alpha_i + \alpha_i^* \right)
	\left( \sum_{i=0}^{n-1} \alpha_i\cdots \alpha_{i+\ell-1}a_{i+\ell-1}^* \cdots \alpha_{i+\ell-k}^*  
	+ \alpha_i\cdots\alpha_{i+k-1}\alpha_{i+k-1}^* \cdots\alpha_{i+k-\ell}^* \right) \\
&= \sum_{i=0}^{n-1} \alpha_i\cdots \alpha_{i+\ell}a_{i+\ell}^* \cdots \alpha_{i+\ell-k+1}^*  
	+ \alpha_i\cdots\alpha_{i+k}\alpha_{i+k}^* \cdots\alpha_{i+k-(\ell-1)}^* \\
	&\qquad+ \alpha_i^*\alpha_i\cdots \alpha_{i+\ell-1}a_{i+\ell-1}^* \cdots \alpha_{i+\ell-k}^*  
	+ \alpha_i^*\alpha_i\cdots\alpha_{i+k-1}\alpha_{i+k-1}^*\cdots\alpha_{i+k-\ell}^* \\
&= \sum_{i=0}^{n-1} \alpha_i\cdots \alpha_{i+\ell}a_{i+\ell}^* \cdots \alpha_{i+(\ell+1)-k}^*  
	+ \alpha_i\cdots\alpha_{i+k}\alpha_{i+k}^* \cdots\alpha_{i+(k+1)-\ell}^* \\
	&\qquad+ \alpha_i\cdots \alpha_{i+\ell-1}a_{i+\ell-1}^* \cdots \alpha_{i+(\ell+1)-k}^*  
	+ \alpha_i\cdots\alpha_{i+k-1}\alpha_{i+k-1}^*\cdots\alpha_{i+(k+1)-\ell}^* \\
&= \OO(\ell+1,k) + \OO(\ell,k+1).
\end{align*}

On the other hand, if $\ell=k$, then
\begin{align*}
\OO(1,0)&\OO(\ell,\ell) 
= \left( \sum_{i=0}^{n-1} \alpha_i + \alpha_i^* \right)
	\left( \sum_{i=0}^{n-1} \alpha_i\cdots \alpha_{i+\ell-1}a_{i+\ell-1}^* \cdots \alpha_i^* \right) \\
&= \sum_{i=0}^{n-1} \alpha_i\cdots \alpha_{i+\ell}a_{i+\ell}^*\cdots \alpha_{i+1}^*
	+ \alpha_i^*\alpha_i\cdots \alpha_{i+\ell-1}a_{i+\ell-1}^* \cdots \alpha_i^* \\
&= \sum_{i=0}^{n-1} \alpha_i\cdots \alpha_{i+\ell}a_{i+\ell}^*\cdots \alpha_{i+1}^*
	+ \alpha_{i+1}\cdots \alpha_{i+\ell}\alpha_{i+\ell}^*\cdots \alpha_i^* \\
&= \OO(\ell+1,\ell).
\end{align*}

For \eqref{eq.O2}, the result is obvious if $m=1$. Suppose it holds for some $m$, then
\begin{align*}
\OO(1,1)^{m+1} &= \OO(1,1)\OO(m,m) \\
	&= \left( \sum_{i=0}^{n-1} \alpha_i\alpha_i^* \right) \left( \sum_{i=0}^{n-1} \alpha_i \cdots \alpha_{i+m-1}\alpha_{i+m-1}^*\alpha_{i+m-2}^* \cdots \alpha_i^* \right) \\
	&= \sum_{i=0}^{n-1} \alpha_i\alpha_i^*\alpha_i \cdots \alpha_{i+m-1}\alpha_{i+m-1}^* \cdots \alpha_i^* \\
	&= \sum_{i=0}^{n-1} \alpha_i \cdots \alpha_{i+m}\alpha_{i+m}^*\cdots \alpha_i^* \\
	&= \OO(m+1,m+1).
\end{align*}
The result now follows by induction.
\end{proof}

Set $s_0=\OO(0,0) = 1$, $s_1 = \OO(1,0)$, and $s_2 = \OO(2,0)$. We claim that $R^{D_n} = \kk[s_1,s_2]$.

\begin{lemma}
The orbit sums $s_1$ and $s_2$ commute.
\end{lemma}
\begin{proof}
We recall first that for every arrow $\alpha$ there is exactly one nonstar arrow $\beta$ and one star arrow $\gamma$ such that $\alpha\beta \neq 0$ and $\alpha\gamma \neq 0$. Using this fact and the preprojective relation we have,
\begin{align*}
s_1s_2 
	&= \left( \sum_{i=0}^{n-1} \alpha_i + \alpha_i^* \right) \left( \sum_{i=0}^{n-1} \alpha_i \alpha_{i+1} + \alpha_i^* \alpha_{i-1}^* \right) \\
	&= \sum_{i=0}^{n-1} \alpha_i\alpha_{i+1}\alpha_{i+2} + \alpha_i\alpha_i^* \alpha_{i-1}^* 
		+ \alpha_i^*\alpha_i\alpha_{i+1} + \alpha_i^*\alpha_{i-1}^*\alpha_{i-2}^* \\
	&= \sum_{i=0}^{n-1} \alpha_i\alpha_{i+1}\alpha_{i+2} + \alpha_{i-1}^*\alpha_{i-1} \alpha_{i-1}^* 
		+ \alpha_{i+1}\alpha_{i+1}^*\alpha_{i+1} + \alpha_i^*\alpha_{i-1}^*\alpha_{i-2}^* \\
	&= \sum_{i=0}^{n-1} \alpha_i\alpha_{i+1}\alpha_{i+2} + \alpha_{i-1}^*\alpha_{i-2}^* \alpha_{i-2} 
		+ \alpha_{i+1}\alpha_{i+2}\alpha_{i+2}^* + \alpha_i^*\alpha_{i-1}^*\alpha_{i-2}^* \\
	&= \left( \sum_{i=0}^{n-1} \alpha_i \alpha_{i+1} + \alpha_i^* \alpha_{i-1}^* \right)\left( \sum_{i=0}^{n-1} \alpha_i + \alpha_i^* \right)
	= s_2s_1. \qedhere
\end{align*}
\end{proof}

\begin{lemma}
\label{lem.gen}
For all $\ell \geq k \geq 0$, $\displaystyle \OO(\ell,k) \in \kk[s_1,s_2]$.
\end{lemma}
\begin{proof}
We already have $\OO(0,0), \OO(1,0), \OO(2,0) \in \kk[s_1,s_2]$. Then
\begin{align*}
s_1^2 &= \left( \sum_{i=0}^{n-1} \alpha_i + \alpha_i^* \right)\left( \sum_{i=0}^{n-1} \alpha_i + \alpha_i^* \right) \\
	&= \left( \sum_{i=0}^{n-1} \alpha_i\alpha_{i+1} + \alpha_i^*\alpha_{i-1}^* \right) 
		+ \left( \sum_{i=0}^{n-1} \alpha_i\alpha_i^* + \alpha_i^*\alpha_i \right) \\
	&= s_2 + 2\OO(1,1).
\end{align*}
Hence, $\OO(1,1) \in \kk[s_1,s_2]$. Suppose inductively that $\OO(\ell,k) \in \kk[s_1,s_2]$ for all $\ell,k$ with $\ell \geq k \geq 0$ and $\ell+k \leq d$ for some $d \geq 2$. First assume that $d$ is even, so that $\OO(\frac{d}{2},\frac{d}{2}) \in \kk[s_1,s_2]$. Then by \eqref{eq.O1}, $\OO(\frac{d}{2}+1,\frac{d}{2})=\OO(1,0)\OO(\frac{d}{2},\frac{d}{2})\in \kk[s_1,s_2]$. Further, since $\OO(\frac{d}{2}+1,\frac{d}{2}-1) \in \kk[s_1,s_2]$, then
\[ \OO\left(\frac{d}{2}+2,\frac{d}{2}-1\right) =
\OO(1,0)\OO\left(\frac{d}{2}+1,\frac{d}{2}-1\right) - \OO\left(\frac{d}{2}+1,\frac{d}{2}\right).\]
By another induction, we have $\OO(\ell,k) \in \kk[s_1,s_2]$ with $\ell+k=d+1$.

Now assume $d$ is odd. Then $d+1$ is even and since $\OO(1,1) \in \kk[s_1,s_2]$, then by \eqref{eq.O2}, $\OO(1,1)^{(d+1)/2} = \OO(\frac{d+1}{2}, \frac{d+1}{2})$. Now the argument proceeds as in the even case.
\end{proof}

We now proceed to our main result for this section.

\begin{theorem}
\label{thm.Dn}
The Auslander map is not an isomorphism for the pair $(R,D_n)$.
\end{theorem}
\begin{proof}
Combining the previous two lemmas there is a surjective map $\kk[s_1,s_2] \to R^{D_n}$. Since both algebras have the same Hilbert series, then it follows that this map is an isomorphism. It now suffices to show that the set 
\[S=\{ e_0,\hdots,e_{n-1},\alpha_0,\hdots,\alpha_{n-1} \}\] 
is a basis for $R$ over $R^{D_n}$. 
That is, $R$ is a rank $2n$ free module over $R^{D_n}$. Then we have 
\[ R \iso \bigoplus_{i=0}^{n-1} \left( e_i R^{D_n} \oplus \alpha_i R^{D_n}\right)\]
as $R^{D_n}$-modules. Since $\alpha_{n-1} R^{D_n} \iso e_0 R^{D_n}(-1)$, then $\End_{R^{D_n}} R$ contains a map of negative degree and so the Auslander map is not an isomorphism for $(R,D_n)$.

First we show that the set $S$ generates $R$ as a $R^{D_n}$-module.
Clearly $R_0 \subset \Span_{R^{D_n}} S$. Moreover, for all $i=0,\hdots,n-1$,
$\alpha_{i-1}^* = e_i (s_1) - \alpha_i (1)$. Hence, $R_1 \subset \Span_{R^{D_n}} S$.

Note that there are exactly three paths of degree 2 for each vertex. Consider the degree 2 paths based at vertex $0$. We have
\begin{align*}
e_0\OO(1,1) = \alpha_0\alpha_0^*, \quad
e_0\OO(2,0) = \alpha_0\alpha_1 + \alpha_{n-1}^*\alpha_{n-2}^*, \quad
\alpha_0\OO(1,0) = \alpha_0\alpha_1 + \alpha_0\alpha_0^*.
\end{align*}
Hence, $\{ \alpha_0\alpha_0^*, \alpha_0\alpha_1, \alpha_{n-1}^*\alpha_{n-2}^* \} \subset \Span_{R^{D_n}} S$. A similar argument for the remaining vertices shows that $R_2 \subset \Span_{R^{D_n}} S$. 

In particular, the above argument shows that $R_2 = S_1R_1^{D_n} + S_0R_2^{D_n}$, which implies that 
$R_2 = R_1R_1^{D_n} + R_0R_2^{D_n}$. Multiplying by $R_1$ on the left gives $R_3 = R_2R_1^{D_n} + R_1R_2^{D_n}$ and by induction, $R_{m+1} = R_mR_1^{D_n} + R_{m-1}R_2^{D_n}$ for all $m$. Thus, $R$ is generated as a right $R^{D_n}$-module by $R_0$ and $R_1$. It follows that $R \subset \Span_{R^{D_n}} S$. That is, $S$ is a generating set for $R$ as an $R^{D_n}$ module.

For independence, we note that for every element of $S$ there is exactly one other element in $S$ with the same source. Hence, it suffices to prove that $e_i R^{D_n} \cap \alpha_i R^{D_n} = \{0\}$. We do this computation for $i=0$ and the other vertices follow similarly.

Suppose $a \in e_i R^{D_n} \cap \alpha_i R^{D_n}$. We may assume without loss of generality that $a$ is homogeneous of degree $d$. Suppose first that $d$ is even. Then there exist scalars $k_i,k_i' \in \kk$ such that
\begin{align*}
a &= e_0 \left( k_0 \OO(d,0) + k_1 \OO(d-1,1) + \cdots + k_{d/2}(d/2,d/2) \right) \\
a &= \alpha_0 \left( k_0' \OO(d-1,0) + k_1' \OO(d-2,1) + \cdots + k_{d/2-1}'(d/2,d/2-1) \right).
\end{align*}
From the second expression, we note that every path summand of $a$ must contain at least one non-starred arrow. Hence, $k_0=0$. But then from the first expression we note that every path summand of $a$ must contain at least one starred arrow, so $k_0'=0$. Continuing in this way, we see that $a=0$.
\end{proof}

\subsection{The $W_n$ case}

In case $n$ is odd, $\eta_{R,G}$ is an isomorphism if and only if $G$ is a proper subgroup of $D_n$. In case $n$ is even, there is one additional instance when $\eta_{R,G}$ fails to be an isomorphism, namely for the subgroup $W_n$ defined as:
\[ W_n = \langle \tau \in D_n : \tau(e_i)=e_i \text{ for some } i=0,\dots,n-1\rangle .\]
That is, $W_n$ is generated by the reflections in $D_n$ that pass through a vertex. If $n$ is odd, then every reflection fixes a vertex, so $W_n$ contains every reflection and $W_n=D_n$. If $n$ is even, only half of the reflections fix a vertex, so $W_n$ is a proper subgroup of $D_n$. Since $W_n$ is of index 2 in $D_n$, $W_n$ is a maximal subgroup of $D_n$.

Throughout this section we assume $n$ is even. Our strategy will be similar to the previous section. The key difference is that the invariant ring is no longer connected graded. In particular, there are exactly twice as many orbits in each graded piece as in the $D_n$ case.

As in the previous section, for $\ell \geq k \geq 0$, set $B_{\ell,k}=Q_\ell Q_k^*\cup Q_k Q_\ell^*$. Then define
\[
B_{\ell,k}^{\even} = e_i B_{\ell,k} \text{ for $i$ even} \quad\text{and}\quad
B_{\ell,k}^{\odd} = e_i B_{\ell,k} \text{ for $i$ odd}.
\]

\begin{lemma}
\label{lem.Blk2}
For any $p\in B_{\ell,k}^{\even}$ (resp. $p\in B_{\ell,k}^{\odd}$), $\displaystyle \cO(p)= B_{\ell,k}^{\even}$ (resp. $\cO(p) = B_{\ell,k}^{\odd}$). 
\end{lemma}
\begin{proof}
This is similar to the proof of Lemma \ref{lem.Blk}. In particular, the $B_{\ell,k}$ partition the paths of $Q$. However, since $g \in W_n$ preserves the parity of the idempotents $e_i$, it follows that $g(B_{\ell,k}^{\even}) \subset B_{\ell,k}^{\even}$. Because $g$ is bijective then in fact we have equality.

It remains to show that we have an equivalence with the orbits. If $p \in B_{\ell,k}^{\even}$, then clearly $\cO(p) \subset B_{\ell,k}^{\even}$. Since $|B_{\ell,k}^{\even}|=\frac{1}{2} B_{\ell,k}$ and $|W_n|=\frac{1}{2} D_n$, then it follows from the argument in Lemma \ref{lem.Blk} that $|B_{\ell,k}^{\even}|=|\cO(p)|$. A similar argument applies to $|B_{\ell,k}^{\odd}|$.
\end{proof}

Set 
\[ \OO(\ell,k)^{\even} = \sum_{p \in B_{\ell,k}^{\even}} p
\quad\text{and}\quad
\OO(\ell,k)^{\odd} = \sum_{p \in B_{\ell,k}^{\odd}} p.\] 
These form a $\kk$-basis for $R^{W_n}$. Thus, $R^{W_n}$ has \emph{total} Hilbert series 
\[ H_{R^{W_n}}^{tot}(t) = \frac{2}{(1-t)(1-t^2)}.\] 
However, since $(R^{W_n})_0 = \kk^2$, then we can also record the matrix-valued Hilbert series. Let $M$ be the $2 \times 2$ matrix defined by
\begin{align*}
M_{0,0} &= \#\{ p \in B_{\ell,k}^{\even} \text{ with target $e_i$, $i$ even} \}, \\
M_{0,1} &= \#\{ p \in B_{\ell,k}^{\even} \text{ with target $e_i$, $i$ odd} \}, \\
M_{1,0} &= \#\{ p \in B_{\ell,k}^{\odd} \text{ with target $e_i$, $i$ even} \}, \\
M_{1,1} &= \#\{ p \in B_{\ell,k}^{\odd} \text{ with target $e_i$, $i$ odd} \}.
\end{align*}
Note that for $p \in B_{\ell,k}$, the parity of the target depends on the source and the parity of $\ell + k$. Hence, it follows that the matrix-valued Hilbert series of $R^{W_n}$ is
\begin{align*}
H_{R^{W_n}}
	&= \begin{pmatrix}1 & 0 \\ 0 & 1\end{pmatrix} + \begin{pmatrix}0 & 1 \\ 1 & 0\end{pmatrix}t
		+ \begin{pmatrix}2 & 0 \\ 0 & 2\end{pmatrix}t^2 +  \begin{pmatrix}0 & 2 \\ 2 & 0\end{pmatrix}t^3 + \cdots \\
	&= \left( I - \begin{pmatrix}0 & 1 \\ 1 & 0\end{pmatrix}t\right)\inv \left( I - \begin{pmatrix}1 & 0 \\ 0 & 1\end{pmatrix}t^2\right)\inv.
\end{align*}

Proofs of the relations in the next lemma are similar to the corresponding proofs in Lemma \ref{lem.osums}.

\begin{lemma}
Let $\bullet,\dagger$ denote opposite parities. 
The orbit sums $\OO(\ell,k)$ satisfy the following relations
\begin{align}
\label{eq.WO1} 
\OO(1,0)^{\bullet}\OO(\ell,k)^{\dagger}
&= 
\begin{cases}
	\OO(\ell+1,k)^{\dagger} + \OO(\ell,k+1)^{\dagger} & \text{ if $\ell+k$ is even and $\ell > k$ } \\
	\OO(\ell+1,k)^{\dagger} & \text{ if $\ell+k$ is even and  $\ell=k$} \\
	\OO(\ell+1,k)^{\bullet} + \OO(\ell,k+1)^{\bullet} & \text{ if $\ell+k$ is odd and $\ell > k$ } \\
	\OO(\ell+1,k)^{\bullet} & \text{ if $\ell+k$ is odd and  $\ell=k$}
\end{cases} \\
\label{eq.WO2} (\OO(1,1)^{\bullet})^m &= \OO(m,m)^{\bullet}.
\end{align}
\end{lemma}



We set $s_0=\OO(0,0)^{\even}$, $s_1=\OO(1,0)^{\even}$, and $s_2=\OO(2,0)^{\even}$. Similarly, we set
$s_0'=\OO(0,0)^{\odd}$, $s_1'=\OO(1,0)^{\odd}$, and $s_2'=\OO(2,0)^{\odd}$. Let $C$ denote the subalgebra of $R^{W_n}$ generated by these elements. Let $Q$ be the following quiver:
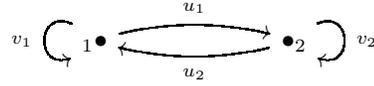
\[  \xymatrix{ 
&  \ar@(ul,dl)[]_{v_1} {}_1 \bullet  \ar@/^0.5pc/[rr]^{u_1}  &  
& \bullet_2 \ar@/^0.5pc/[ll]^{u_2} \ar@(ur,dr)[]^{v_2} & &}
\]
and let $\kk Q$ denote its path algebra. We assign degree 1 to the arrows $u_1,u_2$ and degree 2 to $v_1,v_2$. We will show that $C=R^{W_n}$ and that
$R^{W_n} \iso \kk Q/ (v_1u_1-u_1v_2, v_2u_2-u_2v_1)$.

\begin{remark}
The algebra $\kk Q/ (v_1u_1-u_1v_2, v_2u_2-u_2v_1)$ is a quotient-derivation algebra appearing in the classification of graded twisted Calabi--Yau algebras of global dimension 2 \cite{RR1}. In particular, the matrix corresponding to the Nakayama automorphism $\mu$ is $\left(\begin{smallmatrix}0 & 1 \\ 1 & 0\end{smallmatrix}\right)$ and the $\mu$-twisted superpotential is $v_1u_1-u_2v_1+v_2u_2-u_1v_2$. 
\end{remark}

\begin{lemma}
The relations $s_2s_1 = s_1s_2'$ and $s_2's_1' = s_1's_2$ hold in $C$.
\end{lemma}
\begin{proof}
We prove the first relation. The second is similar.
\begin{align*}
s_2s_1 
	&= \left( \sum_{i=0}^{\frac{n-2}{2}} ( \alpha_{2i}\alpha_{2i+1} + \alpha_{2i+1}^*\alpha_{2i}^* ) \right)
	\left( \sum_{i=0}^{\frac{n-2}{2}} (\alpha_{2i} + \alpha_{2i+1}^*) \right) \\
	&= \sum_{i=0}^{\frac{n-2}{2}} ( \alpha_{2i}\alpha_{2i+1}\alpha_{2i+2} 
		+ \alpha_{2i}\alpha_{2i+1}\alpha_{2i+1}^* 
		+ \alpha_{2i+2} \alpha_{2i+2}^*\alpha_{2i+1}^*
		+ \alpha_{2i+1}^*\alpha_{2i}^*\alpha_{2i-1}^* ) \\
	&= \sum_{i=0}^{\frac{n-2}{2}} ( \alpha_{2i}\alpha_{2i+1}\alpha_{2i+2}
		+ \alpha_{2i}\alpha_{2i}^*\alpha_{2i-1}^* 
		+ \alpha_{2i+2}\alpha_{2i+3}\alpha_{2i+3}^* 
		+ \alpha_{2i+1}^*\alpha_{2i}^*\alpha_{2i-1}^* ) \\
	&= \sum_{i=0}^{\frac{n-2}{2}} ( \alpha_{2i}\alpha_{2i+1}\alpha_{2i+2}
		+ \alpha_{2i}\alpha_{2i}^*\alpha_{2i-1}^* 
		+ \alpha_{2i+1}^*\alpha_{2i+1}\alpha_{2i+2} 
		+ \alpha_{2i+1}^*\alpha_{2i}^*\alpha_{2i-1}^* ) \\
 	&= \left( \sum_{i=0}^{\frac{n-2}{2}} (\alpha_{2i} + \alpha_{2i+1}^*) \right)
	\left(  \sum_{i=0}^{\frac{n-2}{2}} (\alpha_{2i+1}\alpha_{2i+2} + \alpha_{2i+2}^*\alpha_{2i+1}^*) \right)
	= s_1s_2'. \qedhere
\end{align*}
\end{proof}

\begin{lemma}
We have $C=R^{W_n}$.
\end{lemma}
\begin{proof}
Clearly, $C \subset R^{W_n}$. We claim $R^{W_n} \subset C$. By definition, $(R^{W_n})_0 \subset C$ and $(R^{W_n})_1 \subset C$. Now
\begin{align*}
s_1s_1' &= \left( \sum_{i=0}^{\frac{n-2}{2}} (\alpha_{2i} + \alpha_{2i+1}^*) \right) \left( \sum_{i=0}^{\frac{n-2}{2}} (\alpha_{2i+1} + \alpha_{2i}^*) \right) \\
	&= \sum_{i=0}^{\frac{n-2}{2}} (\alpha_{2i}\alpha_{2i+1} + \alpha_{2i}\alpha_{2i}^* + \alpha_{2i+1}^*\alpha_{2i+1} + \alpha_{2i+1}^*\alpha_{2i}^*) \\
	&= \sum_{i=0}^{\frac{n-2}{2}} (\alpha_{2i}\alpha_{2i+1} + \alpha_{2i+1}^*\alpha_{2i}^* + 2\alpha_{2i}\alpha_{2i}^*) \\
	&= s_2 + 2\OO(1,1)^{\even}.
\end{align*}
Thus, $\OO(1,1)^{\even} \in C$. A similar proof with $s_1's_1$ shows that $\OO(1,1)^{\odd} \in C$ so that $(R^{W_n})_2 \subset C$.
The remainder of the proof follows similarly to Lemma \ref{lem.gen} with proper respect shown towards parity. In particular, we use \eqref{eq.WO1}-\eqref{eq.WO2}.
\end{proof}

\begin{theorem}
\label{thm.Wn}
The Auslander map is not an isomorphism for the pair $(R,W_n)$.
\end{theorem}
\begin{proof}
Denote the trivial paths of $Q$ by $f_0,f_1$. There is a map $\phi:\kk Q \to R^{W_n}$ defined by setting
\[ 
f_0 \mapsto s_0, f_1 \mapsto s_0',
u_1 \mapsto s_1, u_2 \mapsto s_1',
v_1 \mapsto s_2, v_2 \mapsto s_2'.
\]
It is easy to verify that this determines a well-defined surjective map $\kk Q \to R^{W_n}$ and \[K=(v_1u_1-u_1v_2, v_2u_2-u_2v_1)\] belongs to $\ker\phi$. By comparing the matrix-valued Hilbert series, it is clear that $\kk Q/K \iso R^{W_n}$.

The remainder of the proof follows analogously to Theorem \ref{thm.Dn}. In particular, $R$ is a free $R^{W_n}$-module with basis $S=\{ e_0,\hdots,e_{n-1},\alpha_0,\hdots,\alpha_{n-1} \}$, and this gives rise to a map in $\End_{R^{W_n}} R$ of negative degree.
\end{proof}

Theorems \ref{thm.Dn} and \ref{thm.Wn} give instances of fixed rings of preprojective algebras which are graded Calabi--Yau. These examples are novel from those presented by Weispfenning in that they do not fix pointwise the degree zero part of $\Pi_{\An}$.

\begin{corollary}
Let $G=D_n$ or $G=W_n$. Then $\p(R,G)=1$.
\end{corollary}
\begin{proof}
Let $p=\alpha_0\alpha_1\cdots \alpha_{n-1}$ and $q=\alpha_{n-1}^*\alpha_{n-2}^* \cdots \alpha_0^*$. Set \[f_1 = e_0 \# 1 + e_0\# r_0 = e_0 (f_G) e_0 \in (f_G).\] Then $(p-q) \# 1 = pf_1 - f_1q \in (f_G)$. Consequently, $\p(R,W_n) \geq 1$. By Theorems \ref{thm.bhz}, \ref{thm.Dn}, and \ref{thm.Wn}, $\p(R,G) < 2$. Thus, $\p(R,G)=1$ by Bergman's Gap Theorem \cite{berg}.
\end{proof}

\noindent \textbf{Acknowledgement}\quad
The authors appreciate the referee's comments and suggestions for improving this manuscript.
\bibliographystyle{plain}

\end{document}